\documentclass[11pt]{amsart}
\usepackage{latexsym, amsmath, amssymb, amsthm, bm, textcase}
\usepackage{amsmath, amsthm, amssymb}
\usepackage[mathscr]{eucal}
\usepackage[T1]{fontenc}
\usepackage[utf8]{inputenc}
\usepackage{mathptmx}
\usepackage{microtype}
\usepackage{verbatim}
\usepackage{algorithm}
\usepackage{url}
\usepackage{algorithmic}
\usepackage[all]{xy}
\usepackage{graphicx}
\usepackage[section]{placeins}
\usepackage{caption}
\usepackage[centering, includeheadfoot, hmargin=1in, vmargin=1in,
  headheight=30.4pt]{geometry}
\allowdisplaybreaks

\floatname{algorithm}{Procedure}
\newcommand{\algref}[1]{Procedure~\ref{#1}}

\theoremstyle{plain}
\newtheorem{thm}{Theorem}[section]
\newtheorem{prop}[thm]{Proposition}
\newtheorem{lemma}[thm]{Lemma}

\theoremstyle{definition}

\newtheorem{ex}[thm]{Example}
\theoremstyle{remark}
\newtheorem{rem}[thm]{Remark}

\newcommand{\propref}[1]{Proposition~\ref{#1}}

\newcommand{\lemmaref}[1]{Lemma~\ref{#1}}

\newcommand{\secref}[1]{Section~\ref{#1}}

\newcommand{\remref}[1]{Remark~\ref{#1}}
\newcommand{\tabref}[1]{Table~\ref{#1}}

\newcommand{\NN}{\mathbb{N}}
\newcommand{\ZZ}{\mathbb{Z}}

\newcommand{\CC}{\mathbb{C}}

\newcommand{\PP}{\mathbb{P}}

\newcommand{\codim}[1]{\operatorname{codim} (#1)}

\newcommand{\Gr}[2]{\operatorname{\mathbb{G}} (#1,#2)}
\newcommand{\td}[1]{\operatorname{td} (#1)}
\newcommand{\ch}[1]{\operatorname{ch} (#1)}
\newcommand{\ppc}[4]{\operatorname{ppc} (#1,#2,#3,#4)}
\newcommand{\ED}[1]{\operatorname{ED} (#1)}
\newcommand{\Pic}[1]{\operatorname{Pic} (#1)}

\renewcommand{\dim}[1]{\operatorname{dim} (#1)}
\renewcommand{\deg}[1]{\operatorname{deg} (#1)}

\begin{document}
\title[]{Numerical Polar calculus and cohomology of line bundles}
\author[S. Di Rocco]{Sandra Di Rocco}
\address{Department of Mathematics, KTH, 100 44 Stockholm, Sweden}
\email{dirocco@math.kth.se}
\urladdr{http://www.math.kth.se/$\sim$sandra}
\author[D. Eklund]{David Eklund}
\address{Department of Mathematics, KTH, 100 44 Stockholm, Sweden} 
\email{daek@math.kth.se}
\urladdr{http://www.math.kth.se/$\sim$daek}
\author[C. Peterson]{Chris Peterson}
\address{Department of Mathematics, Colorado State University, Fort Collins, CO 80523}
\email{peterson@math.colostate.edu}
\urladdr{http://www.math.colostate.edu/$\sim$peterson}

\begin{abstract}
Let $L_1,\dots,L_s$ be line bundles on a smooth variety $X\subset
\mathbb{P}^r$ and let $D_1,\dots,D_s$ be divisors on $X$ such that $D_i$
represents $L_i$. We give a probabilistic algorithm for computing the
degree of intersections of polar classes which are in turn used for
computing the Euler characteristic of linear combinations of
$L_1,\dots,L_s$. The input consists of generators for the homogeneous
ideals $I_X, I_{D_i} \subset \mathbb{C}[x_0,\ldots,x_r]$ defining $X$ and
$D_i$.
\end{abstract}

\maketitle
\section{Introduction}

Let $X \subset \PP^r$ be a smooth $n$-dimensional variety with $n<r$. For
$0\leq j \leq n$ and $V \subseteq \PP^r$ a general linear subspace of
dimension $(r-n-2)+j$, let
$$P_j(X)=\{x \in X: \dim{T_xX \cap V} \geq j-1\}.$$ 
If $X$ has codimension $1$ and
$j=0$, then $V$ is the empty set (rather than a linear subspace) with  the
convention $\dim{\emptyset}=-1.$ Note
that $P_j(X)$ depends on the choice of $V$ even though we have
suppressed this in the notation. However, for general $V$, the class
$[P_j(X)] $ in the Chow ring of the variety  $A_{\ast}(X)$ does not depend on $V$. In fact, for
general $V$, $P_j(X)$ is either empty or of pure codimension $j$ in
$X$ and
$$[P_j(X)]=\sum_{i=0}^j (-1)^i \binom{n-i+1}{j-i}H^{j-i}c_i,$$ where $H
\in A_{n-1}(X)$ is the hyperplane class
and $c_i$ is the $i^{th}$ Chern class of $X.$ In this
setting, $P_j(X)$ is called a $j^{th}$ \emph{polar locus} of $X$ and $[P_j(X)]$ is called the $j^{th}$ \emph{polar class}. 
Because of the close relationship between polar classes and Chern classes, computation of the degrees of intersections of one is equivalent to computation of
the degrees of intersections of the other. In particular the numerical algorithms developed in \cite{DEPS,EBP}
for computing intersection numbers of Chern classes will be used in this note in order to develop an algorithm 
for computing polar degrees and the degrees of intersection of polar classes, \algref{alg:polardegrees} and \algref{alg:polarproducts}. We call these computations ``Numerical Polar Calculus" and will denote the algorithm by {\bf NPC}. Immediate applications of the computation include degree of the discriminant locus and the Euclidean Distance degree, as explained in Subsection \ref{dual}. 
A Macaulay2 algorithm for computing polar classes of (not necessarily smooth or normal) toric varieties has been
previously developed in \cite{HS}.
While applications of NPC are multiple, we focus on uses of the algorithm that 
 we regard as particularly interesting for the applied and computational algebraic geometry community.
In \propref{prop:reduction}  and \algref{alg:general} we illustrate how NPC can be used to develop an algorithm for computing the {\it Euler Characteristic}
of a linear combination of divisors on $X$:
$$\chi(X,a_1D_1+\ldots +a_sD_s)= \sum_{i\geq 0} (-1)^i
\dim{H^i(a_1D_1+\ldots +a_sD_s)}$$ where the $D_i$ are smooth divisors
on $X$ meeting properly. More precisely, we require that any $d \leq
n$ of the $D_i$ meet in a subscheme every component of which has
codimension $d$. The key ingredients are the Hirzebruch-Riemann-Roch
formula and Adjunction formula as explained in Section \ref{intro}.

We briefly illustrate the motivating idea.
The Riemann-Roch theorem for a smooth projective curve $X$ embedded in $\PP^r$ by a line bundle $L$ gives a powerful and striking link between invariants of the line bundle and invariants of the curve:
$$\dim{H^0(X,L)}-\dim{H^1(X,L)}=\chi(X,L)=\deg{L}+1-g$$
where $g=\dim{H^0(X,\omega_X)}$ is the genus of $X.$ 
It can be  reinterpreted as an intersection formula involving the Chern classes of $X$ and $L$:
$$\chi(X,L)=[c_1(L)+\frac{1}{2} c_1(T_X)]\cdot [X]$$
 Using the Chern character approach, Hirzebruch generalized the Riemann-Roch theorem to an expression relating the Euler characteristic of a locally free sheaf $\mathcal E$ on a complex manifold $X$ to the Chern character of $\mathcal E$ and the Todd class of $X$. For a line bundle $L$ on $X$, the Hirzebruch-Riemann-Roch theorem states:
\begin{equation}\label{RR} \sum_{i\geq 0} (-1)^i\dim{H^i(X,L)}=\chi(X,L)=\int_X \ch{L}\td{X}\end{equation}
where $\ch{L}$ is the Chern character of $L$ and $\td{X}$ is the Todd class of $T_X.$ 
The integral over $X$ evaluates the top degree component of the class $\ch{L}\td{X}$ over $X$. 
The Chern character $\ch{L}$ can be expressed as a formal sum of subvarieties of codimension $i$ in $X$
$$\ch{L}=\sum_{i\geq 0}\frac{1}{i!}[L]^i$$
The Todd class can also be  given as a sum, $\td{X}=\td{T_X}=T_0+T_1+\ldots+T_n,$ whose components can be expressed using  Chern classes of $T_X$. The class $T_k$ is a weighted homogeneous polynomial of weighted degree $k$ in the Chern classes $c_1,\ldots,c_k,$ provided $c_i$ has weight $i:$ 
$$\td{T_X}=1+\frac{1}{2}c_1+\frac{1}{12}(c_1^2+c_2)+\frac{1}{24}c_1c_2+\ldots$$
Consequently the right-hand side of the Hirzebruch-Riemann-Roch formula (\ref{RR}) is computable by an algorithm similar to the ones introduced in \cite{DEPS,EBP}.

If $\dim{X}=2$ for instance the formula leads to the
expression $$\chi(X,L)=\frac{L(L-K_X)}{2}+\chi(X,{\mathcal O}_X).$$ Notice that, for simplicity in this paper we use additive notation for the group $\Pic{X}.$ 

When the line bundle is special, i.e. $H^i(X,L)=0$ for $i\geq 1$, the
left hand side of the Hirzebruch-Riemann-Roch formula for the line
bundles $aL$ reduces to $ \chi(X,aL)=\dim{H^0(X, aL)}.$ The graded
algebra $\oplus_{a\geq 0} H^0(X, aL)$ and the Cox-ring, $\oplus_{ L\in
  Pic(X)} H^0(X, L)$, are central in understanding the defining
equations of a variety and also play an important role in birational
geometry (in particular in the Minimal Model Program). For special
line bundles, the algorithm gives a computational method to determine
the dimension of the graded pieces of the ring $\oplus_{a\geq 0}
H^0(X, aL).$ The cohomology vanishing assumption might seem strong but
it is satisfied for important classes like toric varieties and Abelian
varieties. Similarly the Kodaira vanishing theorem assures that, when
$L$ is an ample line bundle, $H^i(X,K_X+aL)=0$ for $i\geq 1$.
 
\noindent{\bf Similar work and future directions.} The degree of the
$j^{th}$ polar variety is equal to the degree of the $j^{th}$ {\it
  coisotropic variety}. The coisotropic varieties $CH_i(X)$ are
hypersurfaces of a Grassmannian variety defined and studied in
\cite{GKZ}. Recent advances can be found in \cite{K17}. Along a
similar line of approach one might consider computing the intersection
formally on a Grassmannian. Consider the Gauss map $\gamma: X
\rightarrow \Gr{n}{r}$, which is generically birational when $X$ is
not a linear space. Instead of intersecting the polar classes on $X$
in order to get intersections of Chern classes of $X$ one could
compute the class of $\gamma(X)$ in the Chow ring of the Grassmannian.
Finally it is important to point out that the method using polar
geometry was introduced in \cite{BL} where the case when the divisor
is a hyperplane section is treated.

 A development of theoretical tools for Polar classes of singular varieties can be found in \cite{P1,P2,P3}. Algorithms and computations of characteristic classes for singular varieties can be found in \cite{Jo, H}. The Euler characteristic algorithm relies on the Hirzebruch-Riemann-Roch formula which for singular varieties has been generalized by Baum-Fulton-MacPherson \cite{BFM} and relies on localized Chern characters.

\noindent{\bf Acknowledgments.} This work benefited from discussions held at Notre Dame and KTH. We are grateful to both institutions 
for generous support. In particular, we thank Andrew Sommese for posing the initial question that led to this note.

\section{Notation and Background}\label{intro}
Throughout this paper $X$ denotes a smooth $n$-dimensional complex
projective variety, $T_X$ denotes its tangent bundle and $c_0, c_1,
\dots, c_n$ denote the Chern classes of the tangent bundle. We use
$c(T_X)=c_0+c_1+\dots+c_n$ for the total Chern class. The Todd class
of $T_X$, also called the Todd class of $X$, is denoted $\td{X}$. For
a vector bundle $E$ on $X$, $\ch{E}$ denotes the Chern character of
$E$.

Given a cycle class $\alpha$ in the intersection ring
$A_{\ast}(X)=\bigoplus_{k=0}^n A_k(X)$, we use $\{\alpha\}_k$ to
denote the projection of $\alpha$ to $A_k(X)$. For a zero-cycle class
$\alpha \in A_0(X)$ represented by $\alpha=\sum_{i=1}^m n_ip_i$ with
$n_i \in \ZZ$ and $p_i \in X$, we define $\int_X \alpha=\sum_{i=1}^m
n_i$ (which is well defined on rational equivalence classes). This is
extended to any $\alpha \in A(X)$ by putting $\int_X \alpha=\int_X
\{\alpha\}_0$. In this paper, by the \emph{degree} of a class we will
mean the degree relative to an embedding of $X$ in $\PP^r$ in the
following sense: if $H \in A_{n-1}(X)$ is the hyperplane class and
$\alpha \in A_k(X)$, then $\deg{\alpha}=\int_X H^k\alpha$. For a subscheme $Z \subseteq X$, we have a corresponding cycle class $[Z] \in A_{\ast}(X)$. 

Let $L$ be a line bundle on $X$.  The Euler characteristic of $L$ is
denoted $\chi(X, L)$ and defined as $$\chi(X, L)=\sum_{i\geq
  0}(-1)^i \dim{H^i(X,L)}.$$ The Hirzebruch-Riemann-Roch (HRR)
formula expresses the Euler characteristic in terms of the Todd class
and the Chern character of $L$:
$$\chi(X, L)=\int_X \ch{L} \td{X}.$$

Consider a smooth divisor $D \subset X$. We will use $D$ to denote
both the divisor and its class in the intersection ring of $X$. In
order to compute the right hand side of the Hirzebruch-Riemann-Roch
formula we will use the following version of the adjunction formula
(see \cite{F} Example 3.2.12):
$$f_*(c(T_D))=c(T_X)\cdot (D-D^2+D^3-\ldots + (-1)^{n+1}D^n),$$
where $f:D \rightarrow X$ is the inclusion.

\section{Numerical Polar calculus}\label{NPC}
In this section we present an algorithm to compute the degree of any
intersection of polar classes of a non singular projective variety. A
useful special case is the degree of the individual polar classes, which we
state in a separate procedure. The algorithms may be implemented using
symbolic software such as Macaulay2 \cite{GS} or numerical algebraic geometry
software such as Bertini \cite{BHSW}.

Let $I=(g_1,\dots,g_t) \subseteq \CC[x_0,\dots,x_r]$ be a homogeneous
ideal defining $X$ and let $l_1,\dots,l_{n-j+2}$ be linear forms
defining $V$. There is a scheme structure on $P_j(X)$ imposed by the
sum of the ideal $I$ and the ideal generated by the $(r-j+2) \times
(r-j+2)$-minors of the Jacobian matrix of
$\{g_1,\dots,g_t,l_1,\dots,l_{n-j+2}\}$. Note that the dimension of
$X$ forces all of the $(r-n+1) \times (r-n+1)$-minors of the Jacobian
matrix of $\{g_1,\dots,g_t\}$ to vanish modulo the ideal $I$. As a
consequence, to determine the scheme structure on $P_j(X)$, it is
enough to take the sum of the ideal $I$ and the ideal generated by the
subcollection of $(r-j+2) \times (r-j+2)$-minors of the Jacobian
matrix of $\{g_1,\dots,g_t,l_1,\dots,l_{n-j+2}\}$ that involve the
final $n-j+2$ rows.

We start with a probabilistic algorithm that returns the equations
defining the power of a polar class. This is then used to
formulate a probabilistic algorithm to compute the degrees of the polar
classes.

\begin{rem} \label{rem:general-polarloci}
In practice we need to intersect general polar loci to compute the
degrees of products of polar classes, so we need to know that general
polar loci intersect properly. It is enough to show that for a purely
$k$-dimensional closed subset $W \subseteq Z$ of a smooth projective
variety $Z \subseteq \PP^r$, every component of $W \cap P_j(Z)$ has
dimension $k-j$ for a general polar locus $P_j(Z)$. This is done in
Lemma~2.2 of \cite{EBP}.
\end{rem}

\begin{algorithm}[ht]
\caption{Power of polar class (ppc)}
\begin{algorithmic} \label{alg:equations}

\REQUIRE Non-negative integers $j,m,n$ with $j \leq n$ and
generators for homogeneous ideal $(g_1,\dots,g_t) \subseteq
\CC[x_0,\dots,x_r]$ defining a smooth variety $Z \subset \PP^r$ of
dimension $n$.

  \ENSURE An ideal $J$ defining a subscheme representing $[P_j(Z)]^m$.

\STATE Let $S = \emptyset$.

\FOR{$1 \leq i \leq m$}

\STATE Let $l_1,\dots,l_{n-j+2} \in \CC[x_0,\dots,x_r]$ be random linear forms.

\STATE Let $S_0$ be the set of $(r-j+2) \times (r-j+2)$-minors of the
Jacobian matrix of $\{g_1,\dots,g_t,l_1,\dots,l_{n-j+2}\}$ that
involve the last $n-j+2$ rows. Let $S=S \cup S_0$.

\ENDFOR

\STATE Let $J$ be the ideal generated by $S$ and $g_1,\dots,g_t$.

\end{algorithmic}
\end{algorithm}

\begin{algorithm}[ht]
\caption{Degrees of polar classes}
\begin{algorithmic} \label{alg:polardegrees}

\REQUIRE Generators for homogeneous ideal $I \subseteq \CC[x_0,\dots,x_r]$ defining a smooth variety $X \subset \PP^r$ of dimension $n$.

\ENSURE The degrees of the polar classes $[P_0(X)],\dots,[P_n(X)]$.

\FOR{$0\leq j \leq n$}

\STATE Let $\deg{[P_j(X)]}=\deg{Y}$ where $Y \subset \PP^r$ is the
subscheme defined by $\ppc{j}{1}{n}{I}$.

\ENDFOR
\end{algorithmic}
\end{algorithm}

\subsection{The dual degree and the Euclidean distance degree}\label{dual}
Polar loci are defined for any variety (possibly singular):
$$P_j(X)=\overline{\{x \in X_{\text{sm}}: \dim{T_xX \cap V} \geq j-1\}}.$$ The degrees of the polar classes can be used to compute two important invariants of a projective variety: the degree of the discriminant locus and the
Euclidean distance degree.
Recall that the irreducible variety 
$$X^*=\overline{\{H\in(\PP^r)^*\text{ such that } H\text{ is tangent to }X \text{ at a smooth point}\}}$$  is called the discriminant locus (or the dual variety) of the embedding $X\hookrightarrow\PP^r.$ Its dimension and degree is given by the polar variety as:
\begin{enumerate}
\item $\codim{X^*}=n+1-{\rm max}\{ k \ | \ P_k(X) \neq \emptyset\}$,
\item If $P_n(X) \neq \emptyset$ then $\deg{P_n(X)}$ is the degree of the
  irreducible polynomial defining the hypersurface $X^*.$
\end{enumerate}

In the case when $X$ is nonsingular $[P_j(X)]=c_j(J_1(L))$ where $L$ is the line bundle defining the embedding and $J_1(L)$ denotes
the first jet bundle. Thus, polar calculus gives information on invariants of jets and on differential properties of the embedding.

Another important invariant that can be expressed in terms of degrees
of polar classes is the Euclidean distance degree (ED degree), see
\cite{DHOST}. In the affine case, this is the number of critical
points of the squared distance function from the variety to a given
(general) point in the ambient space. For projective varieties the ED degree is defined as
the ED degree of its affine cone. The ED degree of a variety $X$, $\ED{X}$, is then equal to the
sum of the degrees of the polar classes of $X$ \cite{DHOST}.

\subsection{ Intersections} We now proceed to the general case of intersection of polar classes.

\begin{algorithm}[ht]
\caption{Degrees of products of polar classes}
\begin{algorithmic} \label{alg:polarproducts}

  \REQUIRE Generators for homogeneous ideal $I \subseteq
  \CC[x_0,\dots,x_r]$ defining a smooth variety $X \subset \PP^r$ of
  dimension $n$.

  \ENSURE An array $V$ containing the degrees of all products
  $$\prod_{j=1}^n [P_j(X)]^{m_j},$$ with $0 \leq m_j \leq n$ and
  $\sum_{j=1}^n jm_j \leq n$.

  \STATE Let $M \subset \NN^{n}$, $M=\{(m_1,\dots,m_n):0 \leq m_j \leq
  n,\; \sum_{j=1}^n jm_j \leq n\}$. Let $V$ be an empty array.

\FOR{$(m_1,\dots,m_n) \in M$}

\STATE Let $K=\sum_{j=1}^n \ppc{j}{m_j}{n}{I}$.

\STATE Let $Y \subseteq \PP^r$ be the subscheme defined by $K$. Adjoin
$\deg{Y}$ to $V$.

\ENDFOR

\end{algorithmic}
\end{algorithm}

\begin{ex} \label{ex:surface-polars}
  Let $M$ be a $2\times 4$-matrix of general linear forms in 5
  variables and consider the ideal $J$ generated by all the $2\times
  2$-minors of $M$. The ideal $J$ defines a smooth rational quartic
  curve $D \subset \PP^4$. If we let $I$ be an ideal generated by
  two general degree 2 elements of $J$, then $I$ defines a smooth
  quartic surface $X$ containing $D$.

  We applied \algref{alg:polarproducts}, implemented
  in Macaulay2 \cite{GS}, to the ideal $I$. The result is:
\[
\begin{array}{l}
\deg{[P_0(X)]}=4, \\
\deg{[P_1(X)]}=8, \\
\deg{[P_2(X)]}=12,\\
\deg{[P_1(X)] \cdot [P_1(X)]}=16,\\
\deg{X^*}=12,\\
\ED{X}=24.
\end{array}
\]
\end{ex}

\section{The general algorithm for HRR} \label{sec:general}
We first illustrate the algorithm by working out the surface case and
then present the general version.

\subsection{Surfaces} \label{sec:surfaces}

For simplicity of notation we treat the case of two divisors. Consider
the case of a smooth surface $X \subset \PP^r$ and smooth curves $D,
E \subset X$ representing line bundles $L_1,L_2$ (with $D, E$ meeting in a zero-scheme). In this case, we
have
$\ch{aL_1+bL_2}=1+(aD+bE)+\frac{1}{2}(a^2D^2+2ab\hskip
1pt DE+b^2E^2)$ and
$\td{X}=1+\frac{1}{2}c_1+\frac{1}{12}(c_1^2+c_2)$. Therefore,
$$\{\ch{aL_1+bL_2}\cdot
\td{X}\}_0=\frac{1}{12}(c_1^2+c_2)+\frac{1}{2}c_1(aD+bE)+\frac{1}{2}(a^2D^2+2abDE+b^2E^2),$$
and the right hand side of the HRR formula is the degree of this
0-cycle class. For $j \in \{1, 2\}$, let $d_j=f_*(c_{j-1}(T_D))$ and
$e_j=g_*(c_{j-1}(T_E))$, where $f:D \rightarrow X, g:E\to X$ are the
inclusions. We will compute the right hand side of the HRR formula by
computing the degrees of the following classes on $X$ in the given
order: $1, c_1, c_2, c_1^2, d_1, e_1, c_1d_1, c_1e_1, d_2, e_2,
d_1e_1$. To see that this is sufficient, first observe that $d_1=D,
e_1=E$ and hence $c_1D=c_1d_1, c_1E=c_1e_1$ and $DE=d_1e_1$. For
$D^2$, we apply the adjunction formula: $f_*(c(T_D))=c(T_X)\cdot
(D-D^2)$. This implies that $d_2=c_1d_1-D^2$, and hence
$D^2=c_1d_1-d_2$. Similarly for $E^2$ we obtain $E^2=c_1e_1-e_2$.

The degrees of the classes $1, c_1, c_2, c_1^2, d_1, e_1,
c_1d_1, c_1e_1, d_2, e_2, d_1e_1$ can
be computed by intersecting polar classes of $X$, $D$ and $E$. Namely,
these numbers can be solved  successively by computing the degrees
of the following classes:
\[
\begin{array}{l}
{[P_0(X)]=1},\\
{[P_1(X)]=3H-c_1},\\
{[P_2(X)]=3H^2-2Hc_1+c_2},\\
{[P_1(X)] \cdot [P_1(X)]=9H^2-6Hc_1+c_1^2},\\
{f_*[P_0(D)]=d_1},\\
{g_*[P_0(E)]=e_1},\\
{[P_1(X)] \cdot f_*[P_0(D)]=3Hd_1-c_1d_1},\\
{[P_1(X)] \cdot g_*[P_0(E)]=3He_1-c_1e_1},\\
{f_*[P_1(D)]=2Hd_1-d_2},\\
{g_*[P_1(E)]=2He_1-e_2},\\
{f_*[P_0(D)]g_*[P_0(E)]=d_1e_1},
\end{array}
\]
where $H \in A_1(X)$ is the hyperplane class.

\subsection{ The general Euler algorithm}
We now describe the general algorithm for computing the Euler
characteristic of multiples of a line bundle. We have implemented it
using Macaulay2 but since the problem is reduced to computing the
number of solutions to a polynomial system, an alternative is to use
software for the numerical solution of such systems. The input
consists of homogeneous ideals $I, J_1,\dots,J_s \subseteq
\CC[x_0,\dots,x_r]$ where $I$ defines a smooth scheme $X \subseteq
\PP^r$ and $J_i$ defines a smooth divisor $D_i \subset X$. If $X$ has
dimension $n$ then any $d \leq n$ of the $D_i$ should meet in a
subscheme every component of which has codimension $d$. The algorithm
is probabilistic as it depends on generic choices of linear subspaces
defining polar loci of $X$ and $D_1,\dots,D_s$.  The following lemma
is an application of the HRR formula and adjunction.

\begin{lemma} \label{lemma:hrr-and-adjunction}
Let $f_i:D_i \rightarrow X, \, i=1,\ldots,s$ be inclusion maps and
let $d_{j,i}={f_i}_*(c_{j-1}(T_{D_i}))$, and $c_j=c_j(T_X)$ for $1
\leq j \leq n$. The Euler characteristic $\chi(X,\sum_{i=1}^s a_iD_i)$
may be expressed in terms of degrees of monomials in $c_1,\dots,c_n,
d_{j,i}$ at most linear in $d_{1,i},\dots,d_{n,i}$ for $i=1,\ldots,
s.$
\end{lemma}
\begin{proof}
By the HRR formula, $\chi(X,\sum_{i=1}^s a_iD_i)$ may be expressed in
terms of the degrees of monomials of the form $c\cdot D_1^{k_1}\cdots
D_s^{k_s}$ where $k_i \in \NN$ and $c \in A_{k_1+\cdots+k_s}(X)$ is a
monomial in $c_1(T_X), \ldots, c_n(T_X)$.

Furthermore, $c\cdot D_1^{k_1}\cdots D_s^{k_s}$ may be expressed in
terms of monomials in $c_1,\dots,c_n,d_{j,i}$ which are (at most)
linear in $d_{1,i},\dots,d_{n,i}$ for $i=1,\ldots,s$. To see this,
note that the adjunction formula for $D_i$ states that $$\sum_{j=1}^n
d_{j,i} = c(T_X)\cdot ({D_i}-{D_i}^2+{D_i}^3-\ldots +(-1)^{n+1}{D_i}^n),$$
which means that $$d_{k,i} = \sum_{j=1}^k (-1)^{j+1}D_i^jc_{k-j}$$ for
$k=1,\dots,n$. Note that $d_{1,i}=D_i$. It follows by induction over $k$
that for $k \geq 1$, $D_i^k$ lies in the subgroup of $A_k(X)$
generated by elements of the form $\mu\cdot d_{j,i}$ where $\mu$ is a
monomial in $c_1,\dots,c_n$.
\end{proof}

\begin{prop} \label{prop:reduction}
The computation of $\chi(X,\sum_{i=1}^s a_iD_i)$ may be reduced to
computing degrees of monomials in the polar classes of $X$ and
$D_1,\dots,D_s$. The monomials are at most linear in the polar classes
of $D_i$.
\end{prop}
\begin{proof}
By \lemmaref{lemma:hrr-and-adjunction}, $\chi(X,\sum_{i=1}^s a_iD_i)$
can be expressed in terms of degrees of monomials in the classes
$c_1,\dots,c_n,d_{j,i}$, at most linear in
$d_{1,i},\dots,d_{n,i}$. The statement is clear once we express the
classes $c_1,\dots,c_n,d_{j,i}$ in terms of polar classes of $X$,
$D_1,\dots,D_s$ and the hyperplane class $H \in A_*(X)$. Recall that
for $0 \leq j \leq n$,
$$[P_j(X)]=\sum_{i=0}^j (-1)^i \binom{n-i+1}{j-i}H^{j-i}c_i.$$
Inverting this relationship gives, for $0 \leq j \leq
n$, $$c_j=\sum_{i=0}^j (-1)^i \binom{n-i+1}{j-i}H^{j-i}[P_i(X)].$$
In the same way, we get for $1\leq i \leq s$ and $0 \leq j \leq n-1$
that
$$d_{j+1,i}=\sum_{l=0}^j (-1)^l
\binom{n-l}{j-l}H^{j-l}{f_i}_*[P_l(D_i)].$$
\end{proof}

By \propref{prop:reduction}, \remref{rem:general-polarloci} and the
assumption regarding proper intersection of the $D_i$, the problem of
computing the Euler characteristic can be reduced to computing the
intersections of certain polar loci of $X$ and $D_i$ and we will state
the algorithm in this form. This is a minor extension of
\algref{alg:polarproducts} to include polar classes of the divisors
$D_i$. To compute the degrees of these intersections we need a routine
that computes the degree of a projective variety. Such a routine will
be assumed given and is used as a building block in the algorithm.
The main algorithm \algref{alg:general} will make use of the
subroutine \algref{alg:equations} (ppc) which returns the equations
defining a power of a polar class.

\begin{algorithm}[ht]
\caption{Euler Characteristics}
\begin{algorithmic} \label{alg:general}

  \REQUIRE Generators for homogeneous ideal $I \subseteq
  \CC[x_0,\dots,x_r]$ defining a smooth variety $X \subset \PP^r$ of
  dimension $n$ and ideals $J_1 \supset I, \dots,J_s \supset I$
  defining smooth divisors $f_i: D_i \rightarrow X$. The divisors are
  assumed to meet properly in the sense that any $d \leq n$ of them meet
  in a subscheme every component of which has codimension $d$.

  \ENSURE An array $V$ containing the degrees of all products
  $$\prod_{i=1}^s {f_i}_*([P_{{k_i}-1}(D_i)])^{a_i} \prod_{j=1}^n
          [P_j(X)]^{m_j},$$ with $0 \leq m_j \leq n$, $1\leq k_i\leq
          n$, $a_i \in \{0,1\}$ and $\sum_{i=1}^s a_ik_i+\sum_{j=1}^n
          jm_j \leq n$.

  \STATE Let $M \subset
  \NN^{n+2s}$, $$M=\{(m_1,\dots,m_n,k_1,\dots,k_s,a_1,\dots,a_s):0
  \leq m_j \leq n, 1\leq k_i \leq n, 0 \leq a_i \leq 1,\sum_{i=1}^s
  a_ik_i+\sum_{j=1}^n jm_j \leq n\}.$$ Let $V$ be an empty array.

\FOR{$(m_1,\dots,m_n,k_1,\dots,k_s,a_1,\dots,a_s) \in M$}

\STATE Let $A=\{i: a_i=1 \}$.

\STATE Let $K=\sum_{j=1}^n
\ppc{j}{m_j}{n}{I}+\sum_{i \in A}\ppc{k_i-1}{a_i}{n-1}{J_i}$.

\STATE Let $Y \subseteq \PP^r$ be the subscheme defined by $K$. Adjoin
$\deg{Y}$ to $V$.

\ENDFOR

\end{algorithmic}
\end{algorithm}

\subsection{Complexity}

In \cite{BL}, the problem of computing the Hilbert polynomial of a
smooth equidimensional projective variety is reduced in polynomial
time to the problem of computing the number of solutions to a
polynomial system. In this paper we have a similar
reduction for computing the Euler characteristic of a
line bundle on a projective variety. The complexity of this reduction
is not the focus of this paper but there are a few important points to
make in this context.

First, computing degrees of polar varieties and their intersections by
forming minors of the Jacobian matrix extended by linear forms leads
to exponential size, which is unnecessary (see \cite{BL} Example
3.11).

Second, the treatment in \algref{alg:general} is redundant since not
all monomials of the form $c\cdot D_1^{k_1}\cdots D_s^{k_s}$ (with
notation as in the proof of \lemmaref{lemma:hrr-and-adjunction})
appear in the right hand side of the HRR formula. 
For example, in the
case of a threefold $X$, $c_3(T_X)$ does not appear at all. A more
efficient algorithm would only compute the monomials that are needed.

\section{A gallery of examples}
\begin{ex} \label{ex:surface}
  Let $M$ be a $2\times 4$-matrix of general linear forms in 5
  variables. Consider the ideal $J$ generated by the $2\times
  2$-minors of $M$. The ideal $J$ defines a smooth rational quartic
  curve $D \subset \PP^4$. If we let $I$ be an ideal generated by
  two general degree 2 elements of $J$, then $I$ defines a smooth
  quartic surface $X$ containing $D$.

  We applied \algref{alg:general} in \secref{sec:general}, implemented
  in Macaulay2 \cite{GS}, to the ideals $I$ and $J$. The result is:
\[
\begin{array}{l}
\deg{[P_0(X)]}=4, \\
\deg{[P_1(X)]}=8, \\
\deg{[P_2(X)]}=12,\\
\deg{[P_1(X)] \cdot [P_1(X)]}=16,\\
\deg{f_*[P_0(D)]}=4,\\
\deg{[P_1(X)] \cdot f_*[P_0(D)]}=8,\\
\deg{f_*[P_1(D)]}=6,\\
\dim{X^*}=3, \deg{X^*}=12,\\
\ED{X}=24,
\end{array}
\]
where $f:D \rightarrow X$ is the inclusion. Let $H$ be the hyperplane
class of $X$. For the Chern classes of $X$, the table above and the
discussion in \secref{sec:surfaces} give
$\deg{c_1}=\deg{3H-[P_1(X)]}=12-8=4$ and
$\deg{c_2}=\deg{[P_2(X)]-3H^2+2Hc_1}=12-12+2\cdot 4=8$. Moreover,
$\deg{c_1^2}=\deg{[P_1(X)]^2-9H^2+6Hc_1}=16-9\cdot4+6\cdot 4=4$. Let
$d_j=f_*(c_{j-1}(T_D))$ for $j \in \{1,2\}$. Then $\deg{d_1}=4$ and
$\deg{c_1d_1}=\deg{3Hd_1-[P_1(X)] \cdot f_*[P_0(D)]}=12-8=4$. Also,
$\deg{d_2}=\deg{2Hd_1-f_*[P_1(D)]}=8-6=2$. Finally, we have that
$\deg{c_1D}=\deg{c_1d_1}=4$ and by the adjunction formula
$\deg{D^2}=\deg{c_1d_1-d_2}=4-2=2$. Putting this together we get that
for $a \in \ZZ$:
$$\chi(X,
aD)=\frac{1}{12}(4+8)+\frac{1}{2}4a+\frac{1}{2}2a^2=1+2a+a^2.$$ 

\end{ex}

\begin{ex} \label{ex:surface2}
For this example we consider the Veronese embedding $$\nu: \PP^2
\rightarrow \PP^5:(x,y,z) \mapsto (x^2,y^2,z^2,xy,xz,yz).$$ The image
$X=\pi(\nu(\PP^2))$ of the projection $$\pi:\PP^5 \rightarrow
\PP^4:(a,b,c,d,e,f) \mapsto (a+c,b+c,d,e,f)$$ is smooth. Consider
a divisor $D \subset X$ corresponding to a general cubic curve in
$\PP^2$. The equations for $X$ and $D$ embedded in $\PP^4$ can be
found via elimination, let the corresponding ideals be denoted $I$ and
$J$. We will consider the Euler characteristic of the line bundle
corresponding to $aD+bH$, where $H \in A_1(X)$ is the hyperplane class
and $a,b \in \ZZ$.

Running \algref{alg:general} in \secref{sec:general} on $I$ and $J$, results in the following:

\[
\begin{array}{l}
\deg{[P_0(X)]}=4, \\
\deg{[P_1(X)]}=6, \\
\deg{[P_2(X)]}=3,\\
\deg{[P_1(X)] \cdot [P_1(X)]}=9,\\
\deg{f_*[P_0(D)]}=6,\\
\deg{[P_1(X)] \cdot f_*[P_0(D)]}=9,\\
\deg{f_*[P_1(D)]}=12,\\
\dim{X^*}=3, \deg{X^*}=3,\\
\ED{X}=13.
\end{array}
\]

Solving for the degrees of the needed monomials in the Chern classes
and $D$, we get the result in \tabref{tab:result-surf}.

\begin{table}[ht]
  \caption{}
\begin{tabular}{llllllllll} \label{tab:result-surf}
& $H^2$ & $c_1$ & $c_2$ & $c_1^2$ & $d_1$ & $c_1d_1$ & $d_2$ & $c_1D$ & $D^2$ \\
degree & 4 & 6 & 3 & 9 & 6 & 9 & 0 & 9 & 9 \\
\hline
\end{tabular}
\end{table}

For the Euler characteristic of $aD+bH$ we get:

$$\chi(X,aD+bH) = \frac{1}{12}(c_1^2+c_2)+\frac{1}{2}c_1(aD+bH)+\frac{1}{2}(aD+bH)^2 = 1 + \frac{9}{2}a + 3b + \frac{9}{2}a^2 + 6ab + 2b^2.$$

\end{ex}

\begin{ex}

 Let $J \subseteq \CC[x_0,\dots,x_5]$ be generated by three general
 forms of degree 2 and let $D$ be the corresponding complete
 intersection surface. If we let $I$ be an ideal generated by two
 general degree 2 elements of $J$, then $I$ defines a smooth quartic
 threefold $X$ containing $D$.

Running \algref{alg:general} on $I$ and $J$ gives the following
result:
\[
\begin{array}{l}
\deg{[P_0(X)]}=4,\\
\deg{[P_1(X)]}=8,\\
\deg{[P_2(X)]}=12,\\
\deg{[P_3(X)]}=16,\\
\deg{[P_1(X)] \cdot [P_1(X)]}=16,\\
\deg{[P_1(X)] \cdot [P_2(X)]}=24,\\
\deg{f_*[P_0(D)]}=8,\\
\deg{[P_1(X)] \cdot f_*[P_0(D)]}=16,\\
\deg{[P_2(X)] \cdot f_*[P_0(D)]}=24,\\
\deg{[P_1(X)] \cdot [P_1(X)] \cdot f_*[P_0(D)]}=32,\\
\deg{f_*[P_1(D)]}=24,\\
\deg{[P_1(X)] \cdot f_*[P_1(D)]}=48,\\
\deg{f_*[P_2(D)]}=48.
\end{array}
\]
Solving for the degrees of the needed monomials in the Chern classes
and $D$ as above we get the result in \tabref{tab:result}. 

\begin{table}[ht]
  \caption{}
\begin{tabular}{llllllllllllllll} \label{tab:result}
& $c_1$ & $c_2$ & $c_1^2$ & $c_1c_2$ & $d_1$ & $c_1d_1$ & $c_2d_1$ & $c_1^2d_1$ & $d_2$ & $c_1d_2$ & $d_3$ & $c_1^2D$ & $c_2D$ & $c_1D^2$ & $D^3$\\
degree & 8 & 12 & 16 & 24 & 8 & 16 & 24 & 32 & 0 & 0 & 24 & 32 & 24 & 32 & 32\\
\hline
\end{tabular}
\end{table}

This results in the Euler characteristic 

$$\chi(X, aD)=1+\frac{14}{3}a+8a^2+\frac{16}{3}a^3.$$
Notice that the algorithm gives $\dim{X^*}=4, \deg{X^*}=16$ and $\ED{X}=40.$
\end{ex}

\begin{ex}
Let $M$ be a $2\times 3$-matrix of general linear forms in
$\CC[x_0,\ldots,x_5]$ and consider the variety $X$ defined by the
$2\times 2$-minors of $M$. The threefold $X$ is isomorphic to $\PP^1
\times \PP^2$ and we may consider a general divisor $D$ of type
$(1,2)$. To get equations for $D$ one can pick a random
multihomogeneous polynomial $p$ in $\CC[a,b,x,y,z]$ of type $(1,2)$,
consider the graph of the Segre embedding $\PP^1 \times \PP^2
\rightarrow \PP^5$ intersected with $\{p=0\}$, and eliminate the
variables $a,b,x,y,z$.

Running \algref{alg:general} on the ideals defining $X$ and $D$ gives
the following result:
\[
\begin{array}{l}
\deg{[P_0(X)]}=3,\\
\deg{[P_1(X)]}=4,\\
\deg{[P_2(X)]}=3,\\
\deg{[P_3(X)]}=0\\
\deg{[P_1(X)] \cdot [P_1(X)]}=5,\\
\deg{[P_1(X)] \cdot [P_2(X)]}=3,\\
\deg{f_*[P_0(D)]}=5,\\
\deg{[P_1(X)] \cdot f_*[P_0(D)]}=7,\\
\deg{[P_2(X)] \cdot f_*[P_0(D)]}=6,\\
\deg{[P_1(X)] \cdot [P_1(X)] \cdot f_*[P_0(D)]}=9,\\
\deg{f_*[P_1(D)]}=10,\\
\deg{[P_1(X)] \cdot f_*[P_1(D)]}=14,\\
\deg{f_*[P_2(D)]}=12.
\end{array}
\]
Solving for the degrees of the needed monomials in the Chern classes
and $D$ as above we get the result in \tabref{tab:result2}.

\begin{table}[ht]
  \caption{}
\begin{tabular}{llllllllllllllll} \label{tab:result2}
& $c_1$ & $c_2$ & $c_1^2$ & $c_1c_2$ & $d_1$ & $c_1d_1$ & $c_2d_1$ & $c_1^2d_1$ & $d_2$ & $c_1d_2$ & $d_3$ & $c_1^2D$ & $c_2D$ & $c_1D^2$ & $D^3$\\
degree & 8 & 9 & 21 & 24 & 5 & 13 & 15 & 33 & 5 & 13 & 7 & 33 & 15 & 20 & 12\\
\hline
\end{tabular}
\end{table}

This results in the Euler characteristic

$$\chi(X, aD)=1+4a+5a^2+2a^3.$$

Notice that $\ED{X}=10$ and $\dim{X^*}=3$ since $\deg{[P_3(X)]}=0$ and
$\deg{[P_2(X)]}=3$.
\end{ex}

\end{document}